\documentclass{amsart}

\usepackage{graphicx}
\usepackage[all,knot,arc,import,poly]{xy}
\usepackage{float}

\newtheorem{theorem}{Theorem}[section]
\newtheorem{lemma}[theorem]{Lemma}
\newtheorem{prop}[theorem]{Proposition}
\newtheorem{cor}[theorem]{Corollary}
\newtheorem{rmk}[theorem]{Remark}
\newtheorem{definition}[theorem]{Definition}
\newtheorem{example}[theorem]{Example}

\newcommand{\Ann}{\operatorname{Ann}}
\newcommand{\Hilb}{\operatorname{Hilb}}
\newcommand{\codim}{\operatorname{codim}}

\def\J{\mathcal{J}}

\def\C{\mathbb{C}}

\def\P{\mathbb{P}}
\def\K{\mathbb{K}}
\def\rk{\operatorname{rk}}

\def\Hess{\operatorname{Hess}}
\def\hess{\operatorname{hess}}

\numberwithin{equation}{section}

\begin{document}

\title{The Jordan type of graded Artinian Gorenstein algebras}

\author{Barbara Costa}
\address{UFRPE}
\email{barbaracostasilva@ufrpe.br}

\author{Rodrigo Gondim}
\address{UFRPE}
\email{rodrigo.gondim@ufrpe.br}


\date{}



\begin{abstract}
We study the general Jordan type of standard graded Artinian Gorenstein algebras, it is a finer invariant than Weak and Strong Lefschetz properties for those algebras. We prove that their Jordan types are determined by 
the rank of certain Mixed Hessians. We give a description of the possible Jordan types for algebras of low socle degree and low codimension. 

\end{abstract}

\maketitle

\section{Introduction}

The Lefschetz properties for Artinian $\K$-algebras are algebraic abstractions inspired by the Hard Lefschetz theorem on the cohomology of smooth projective varieties (see \cite{La} or \cite[Chapter 7]{Ru}). Nowadays there are lots of contexts where the Lefschetz properties have been introduced, 
for instance Kahler manifolds, convex polytopes, Coxeter groups and tropical varieties (see \cite{Ka,Be,HL,KN,NW,St,St2,BN,GZ}). These structures share some algebraic properties, in particular the existence of a certain cohomology ring. 
In various contexts this cohomology ring is a standard graded Artinian $\K$-algebra satisfying some kind of  Poincar\'e duality:
$$A=\bigoplus_{i=0}^dA_i,\ \ \text{with} \ \ A_d \simeq \K\ \ \text{and}\ \ A_{d-i} \simeq \operatorname{Hom}(A_i,A_d)\ \ i=0,\ldots,\lfloor \frac{d}{2}\rfloor.$$
Where the isomorphisms are given by the fact that the restriction of the multiplication of the algebra in complementary degree, that is $A_i \times A_{d-1} \to A_d \simeq \K$, is a perfect pairing. 
One know that a standard graded Artinian $\K$-algebras 
satisfy the Poincar\'e duality if and only if it is a Gorenstein algebra (see \cite{MW, Ru}).\\

For standard graded Artinian Gorenstein $\K$ algebras $A$, the Strong Lefschetz property (SLP) means that there is $l \in A_1$ such that the multiplication maps $\mu_{l^{d-2k}}:A_k \to A_{d-k} $ is an isomorphism for all $k$. The Weak Lefschetz property (WLP) means that there is $l \in A_1$ such that the multiplication maps $\mu_{l}:A_k \to A_{k+1} $ have maximal rank for all $k$. In the algebraic context the Lefschetz properties is also a very important theme of research the last decades (see, for instance, \cite{BI, BL, BMMNZ, H-W,MN1,MN2, St,St2,HMNW,MW,Go,GZ}).

The Jordan type $\mathcal{J}_A$ of a standard graded Artinian Gorenstein $\K$-algebra\\ $A=\displaystyle \bigoplus_{i=0}^dA_i$ is the partition of $N=\dim_{\K}A$ given by the Jordan blocks of the multiplication map $\mu_l:A \to A$ for a generic linear form $l \in A_1$ (see \cite[Chapter 5]{H-W}). 
Notice that this $\K$-linear map is nilpotent, so the Jordan blocks have only $0$ in the diagonal. {\it A priori} the Jordan type depends on $l \in A_1$ but considering $\operatorname{char}(\K)=0$ it is invariant in a open (Zariski) subset of $A_1$ (see \cite[Chapter 5]{H-W} and \cite{IMM}).\\

The Strong Lefschetz properties (SLP) and the Weak Lefschetz property (WLP) of $A$ can be expressed in terms of its Jordan type $\mathcal{J}_A$. Indeed, the algebra $A$ has the WLP if and only if 
the number of parts of $\mathcal{J}_A$ is the Sperner number of $A$, {\it i.e.}, the maximum value of its Hilbert vector $\operatorname{Hilb}(A)$. Moreover, $A$ has the SLP if and only if the partition $\mathcal{J}_A$ is dual, in the terms of Ferrer diagrams, of the partition of $\dim_{\K} A$ given by the Hilbert vector of $A$, by abuse of notation also denoted $\Hilb(A)$. In other words, $A$ has the SLP if and only if $\mathcal{J}_A = \operatorname{Hilb}(A)^{\vee}$ (see \cite{H-W}).  \\

In order to study the Jordan type of an algebra we use mixed Hessians. Higher order Hessians have been introduced in \cite{MW} to control the SLP and used in \cite{Go,GZ} to produce series of algebras failing SLP or WLP. 
More recently in \cite{GZ2} mixed Hessians have been introduced to control both SLP and WLP, they are a generalization of higher order Hessians. We want highlight that there are important recent works in the study of Jordan types (see \cite{IMM,AIK}).\\

In this paper we recall, in the second section, basic facts about Standard graded Artinian Gorenstein algebras, AG algebras for short , including Macaulay-Matlis duality and the Lefschetz properties. In the 
third section we introduce the mixed Hessians and recall the interaction between them and the Lefschetz properties. In the 4th section, we prove our main result which says that the Jordan type of an AG algebra $A=Q/\Ann_f$ depends only on the rank of the mixed Hessians of $f$ (see Theorem \ref{thm:main}). We also give an algorithm to calculate the Jordan type of an explicit AG algebra given in the form $A = Q/\Ann_f$. In the last section we use the previous results to classify the possible Jordan types of AG algebras in low codimension and low socle degree (see Propositions \ref{prop:cubicsingeneral}, \ref{prop:quarticsingeneral} and \ref{prop:quintics}; Theorems \ref{thm:cubics} and \ref{thm:quartics} and Corollaries \ref{cor:wlp4} and  \ref{cor:quinticsWLP}).

\section{Standard graded Artinian Gorenstein algebras}

Let $\K$ be a field of $\operatorname{char}(\K)=0$ and let $A = \displaystyle \bigoplus_{i=0}^d A_i$ be an Artinian $\K$-algebra with $A_d \neq 0$, we say that $A$ is standard graded if $A_0 =\K$ and $A$ is generated in degree $1$ as algebra. 
The Hilbert function (or vector) of $A$ is $$\operatorname{Hilb}(A) = (\dim A_0, \dim A_1, \ldots, \dim A_d).$$
\begin{definition}\rm
A standard graded algebra $A$ is Gorenstein if and only if $\dim A_d = 1$ and the restriction of the multiplication of the algebra in complementary degree, that is, $A_i \times A_{d-1} \to A_d$ are a perfect paring for $i =0,1,\ldots,d$ (see \cite{MW}). 
\end{definition}
A very useful way to produce standard graded Artinian Gorenstein algebras is by using Macaulay-Matlis duality, let us recall this construction.
Let $f \in R = \K[x_1,\ldots,x_N]_d$ be a form of degree $\deg(f)=d \geq 1$ and let 
$Q=\K[X_1,\ldots,X_N]$ be the ring of differential operators associated to $R$. 
We define the Annihilator ideal 
$$\Ann_f = \{\alpha \in Q| \alpha(f)=0\}\subset Q.$$
The homogeneous ideal $\Ann_f$ of $Q$ is also called Macaulay dual of $f$. We define $$A=\frac{Q}{\Ann_f}.$$
One can verify that $A$ is a standard graded Artinian Gorenstein $\K$-algebra such that $A_j=0$ for $j>d$ and such that $A_d \neq 0$ (see \cite[Section 1,2]{MW}). 
We assume, without loss of generality, that $(\Ann_f)_1=0$. 

Conversely, by the Theory of Inverse Systems, we get the following characterization of standard graded
Artinian Gorenstein $\K$-algebras. 

\begin{theorem}{\bf \ ( Double annihilator Theorem of Macaulay)} \label{G=ANNF} \\
Let $R = \K[x_1,\ldots,x_n]$ and let $Q = \K[X_1,\ldots, X_n]$ be the ring of differential operators. 
Let $A= \displaystyle \bigoplus_{i=0}^dA_i = Q/I$ be an Artinian standard graded $\mathbb K$-algebra. Then
$A$ is Gorenstein if and only if there exists $f\in R_d$
such that $A\simeq Q/\Ann_f$.
\end{theorem}

A proof of this result can be found in \cite[Theorem 2.1]{MW}.

\begin{definition} \rm With the previous notation, let $A= \displaystyle \bigoplus_{i=0}^dA_i = Q/I$ be an Artinian Gorenstein $\K$-algebra with $I = \Ann_f$, $I_1=0$ and $A_d \neq 0$. The socle degree of $A$ is $\text{Socle Deg}(A) = d$ which coincides with the degree of the form $f$. 
 The codimension of $A$ is the codimension of the ideal $I \subset Q$ which coincides with its embedding dimension, that is, $\codim A = n$.
\end{definition}

We now recall the so called Lefschetz properties for Standard graded Artinian Gorenstein $\K$-algebras.

\begin{definition}\rm With the previous notation, for $a \in A$ we define the $\K$-linear map $\mu_a:A \to A$ by $\mu_a(x)=ax$.
\begin{enumerate}
 \item[(i)] We say that $A$ has the Strong Lefschetz property (SLP) if there is $l \in A_1$ such that the $\K$-linear multiplication maps $\mu_{l^{d-2i}}:A_i \to A_{d-i}$ are isomorphisms for $i=1,\ldots,\lfloor \frac{d}{2}\rfloor$. In this case $l$ is an Weak Lefschetz element.
 \item[(ii)] We say that $A$ has the Strong Lefschetz property (SLP) if there is $l \in A_1$ such that the $\K$-linear multiplication maps $\mu_{l}:A_i \to A_{i+1}$ are of maximal rank for $i=0,\ldots,d$. In this case $l$ is a Strong Lefschetz element.
\end{enumerate}
 
\end{definition}


Let $A=\displaystyle \bigoplus_{i=0}^dA_i$ be a standard graded Artinian $\K$-algebra and let $M=\displaystyle \bigoplus_{j=0}^m M_j$ be a graded $A$-module. For $l \in A_1$ consider the map $\mu_l:M \to M$ given by $\mu_l(x)=lx$.
Since $l^{d+1}=0$ and $\mu_l^k=\mu_{l^k}$, we conclude that $\mu_l$ is a nilpotent $\K$-linear map whose eigenvalues are only $0$. The Jordan decomposition of such a map is given by Jordan blocks with $0$ in the diagonal, therefore it induces 
a partition of $\dim_{\K}M$ that we denote $\mathcal{J}_{M,l}$. Without loss of generality we consider the partition in a non decreasing order. \\ 

\begin{definition}\rm
 Given a nilpotent linear homomorphism $\mu_l:M \to M$ there is a direct sum decomposition as $\K$-linear spaces $M = \displaystyle \bigoplus_{j=0}^m M_j$ into cyclic $\mu_l$-invariant subspaces. We call 
 a $\mu_l$-invariant basis of these cyclic subspaces $M_i$,  $<v_i, lv_i, \ldots, l^{k_i-1}v_i>$ a string of length $k_i = \dim M_i$. The partition $\mathcal{J}_{M,l}$ is given by the length of the strings $k_i=\dim_{\K}M_i$.
\end{definition}

\begin{definition}\rm Given a partition $P = p_1\oplus \ldots\oplus p_s$ of $N$ with $p_1 \geq \ldots \geq p_s $, we denote $P^ {\vee}$ the dual partition obtained from $P$ exchange rows and columns in the Ferrer diagram (diagram of dots). If $P' = p'_1 \oplus \ldots \oplus p'_t$ is another partition of $N$ with $p'_1 \geq \ldots \geq p'_t $, we
say that $P$ is less than or equal to $P'$ ( $P \preceq P'$) if either:
\begin{enumerate}
 \item[(i)] $s < t$ or 
 \item[(ii)] $s=t$, $p_i=p'_i$ for $i=1,\ldots, j-1$ and $p_j \leq p'_j$ for some $j \leq s=t$. 
 \end{enumerate}
 If the partition $P$ have repeated terms, say $f_1, f_2, \ldots, f_r$ with multiplicity $e_1, e_2, \ldots, e_r$ respectively, in this case we write $$P = f_1^{e_1} \oplus \ldots \oplus f_r^{e_r}.$$ 
\end{definition}

If $\K$ is a field of $\operatorname{char}(\K)=0$, there is a Zariski open non empty subset of $\mathcal{U} \subset A_1$ where $\mathcal{J}_{M,l}$ is constant for $l \in \mathcal{U}$, we call it the Jordan type of $M$ and we denote it $\mathcal{J}_M$ (see \cite[Chapter 5]{H-W}).
We are interested in the Jordan type of $A$ as a module over itself, $\mathcal{J}_A$, in the case that $A$ is a standard graded Artinian Gorenstein $\K$-algebra.\\

The following result is well known, a proof can be found in \cite[Proposition 3.6]{H-W}. According to this result, the WLP can be described by the number of blocks of a Jordan decomposition 
of $\mu_l$.

\begin{prop}\label{prop:jordanWLP}
 Suppose that $A$ is a standard graded Artinian $\K$-algebra. Then $A$ has the WLP if and only if the number of parts of $\mathcal{J}_A$ is equal to the maximum value of its Hilbert vector 
 (called the Sperner number of $A$).
\end{prop}

The following proposition is a special case of \cite[Proposition 3.64]{H-W}. It shows that SLP can be described by the Jordan type of $A$.

\begin{prop}\label{prop:jordanSLP}
 Suppose that $A=\displaystyle \bigoplus_{i=0}^dA_i$ is a standard graded Artinian $\K$-algebra with $A_d \neq 0$ and let $l \in A_1$ for $k >0$. Then:
 \begin{enumerate}
  \item If $\Hilb(A)$ is unimodal, then $\mathcal{J}_{A,l} \preceq \Hilb(A)^{\vee}$; 
  \item The following conditions are equivalent:
  \begin{enumerate}
   \item $l$ is a Strong Lefschetz element; 
   \item $\mathcal{J}_{A,l} = \Hilb(A)^{\vee}$.
  \end{enumerate}

 \end{enumerate}
In particular, if $A$ has the SLP, then $\mathcal{J}_{A} = \Hilb(A)^{\vee}$.
\end{prop}


\section{Hessians and its ranks}

Let let $R=\K[x_1,\ldots,x_n]$ be a polynomial ring in $n$ variables over a field of characteristic zero and let $Q=\K[X_1, \ldots, X_n]$ be the associated ring of differential operators. Let $f \in R_d$ and let $A = \displaystyle \bigoplus_{k=0}^d A_k \simeq Q/\Ann_f$ be  
the standard graded Artinian Gorenstein algebra associated to $f$,  $\text{Socle Deg}(A)=d$  and $\codim(A) =n$. \\

Let $k\le t$ two integers, take $l \in A_1$ and let us consider the $\K$-linear map
    $$\mu_{l^{t-k}}: A_k\to A_t,\,\,\mu_{l^{t-k}}(\alpha)=l^{t-k}\alpha.$$
Let $\mathcal{B}_k=(\alpha_1,\ldots,\alpha_r)$ be a basis of $A_k$ and 
$\mathcal{B}_t=(\beta_1,\ldots,\beta_s)$ be a basis of $A_t.$

\begin{definition}\rm
We call mixed Hessian of $f$ of mixed order $(k,t)$ with respect to the basis $\mathcal{B}_k$ and $\mathcal{B}_t$ the matrix: 
  $$\Hess_f^{(k,t)}:=[ \alpha_i\beta_j(f)]$$
Moreover, we define $\Hess_f^k=\Hess_f^{(k,k)}$, $\hess_f^k = \det(\Hess_f^k)$ and $\hess_f=\hess_f^1$.

  \end{definition}

The dual mixed Hessian $\Hess_f^{(t^*,k)}$ is the Hessian with respect to dual basis $\mathcal{B}^*_t$, therefore:  
$$\rk \Hess_f^{(t^*,k)} = \rk \Hess^{(d-t,k)} \rk \Hess^{(k,d-t)}.$$
 The following result is Theorem \cite[Theorem 2.4]{GZ2} that generalizes \cite[Theorem 4]{Wa1} and \cite[Theorem 3.1]{MW}.

\begin{theorem}\cite{GZ2}\label{thm:generalization}
With the previous notation, let $M$ be the matrix associated to the map $\mu_l:A_k \to A_t$ with respect 
to the bases $\mathcal{B}_k$ and $\mathcal{B}_t$. Let $l = a_1x_1+\ldots+a_nx_n$, and define $l^{\perp} = (a_1:\ldots:a_n)$. Then
    $$M=(t-k)!\Hess_f^{(t^*,k)}(l^{\perp}).$$
    In particular, $\rk M = \rk \Hess_f^{(k, d-t)}$.
\end{theorem}

\begin{example}\rm Consider the algebra associated to the form $f= xuv^3+yu^3v \in \C[x,y,u,v]$, it is a small simplification of an Example due to Ikeda (see \cite{MW, Ik}). 
Its Hilbert vector is $\Hilb(A) = (1,4,7,7,4,1)$, and its  second degree part is \\$A_2 = <XU, XV, YU, YV, U^2, UV, V^2>$. The second Hessian of $f$ is 
$$\Hess^2_f=\left[ \begin{array}{ccccccc}
0& 0 & 0 & 0 & 0 & 0 & 6v \\
0& 0 & 0 & 0 & 0 & 6v & 6u \\
0& 0 & 0 & 0 & 6v & 6u & 0 \\
0& 0 & 0 & 0 & 6u & 0 & 0 \\
0 & 0 & 6v & 6u & 0 & 6y & 0\\                     
0 & 6v & 6u & 0 & 6y & 0 & 6x\\
6v & 6u & 0 & 0 & 0 & 6x & 0
  \end{array} \right]
$$
 It is easy to see that $\hess_f^2=0$. Since $\rk (\Hess^2_f)= \rk (\Hess^{(3^*,2)}_f) \leq 6 < 7$, by Theorem \ref{thm:generalization}, we conclude that for all $l \in A_1$, the multiplication map 
 $\bullet l : A_2 \to A_3$ is not an isomorphism. Therefore, $A$ fails WLP.

 \end{example}

Now we want to show that the rank of the first Hessian can be ``small'' with respect to the codimension $n$ (see Corollary \ref{cor:droprank} for a precise statement). This rank is the rank of $\mu_{l^{d-2}}:A_1 \to A_{d-1}$ for $l \in A_1$ general (Theorem \ref{thm:generalization}). 
The ideas here developed have been introduced in \cite{GRu} for cubics, but is evident there that the construction works for arbitrary degree. 

The following Definition and Proposition are part of classical Gordan-Noether theory. It can be find in \cite[Appendix A]{Go}. 

\begin{definition}\rm \label{def:bigraded} Let $f \in \K[x_1,\ldots,x_n,u_1,\ldots,u_m]$ with $n > m$ be a bihomogeneos polynomial of bidegree $(1,d-1)$. We say that $f$ is a Perazzo polynomial if
$$f= x_1g_1 +\ldots+x_ng_n.$$ 
With $g_i \in \K[u_1,\ldots,u_m]$ linearly independent and not defining a cone in $\P^{m-1}$.
\end{definition}

\begin{prop}
 Perazzo polynomials have vanishing Hessian. 
\end{prop}

\begin{proof}
 By Gordan-Noether Theorem (see \cite{GN, Ru, Go}), an algebraic dependence among the partial derivatives of $f$ implies that $\hess_f=0$ (see \cite[Proposition 3.10]{Go}). 
 Since $\frac{\partial{f}}{{\partial x_i }} = g_1 \in \K[u_1,\ldots,u_m]$ and since $n > m$, then 
$\hess_f=0$. Moreover $\rk \Hess_f = n-\delta$ where $\delta$ is the number of independent equations of algebraic dependence among the derivatives.
\end{proof}

\begin{rmk}\rm 

We recall that in $\P^3$ the only hypersurfaces with vanishing Hessian are the cones and for $n \geq 4$ for each degree $d \geq 3$ there is $f \in \K[x_1,\ldots,x_n]$ 
such that $X=V(f) \subset \P^{n-1}$ is not a cone, but $\hess_f = 0$. The bigraded polynomials here called Perazzo polynomials were introduced by Gordan-Noether and by Perazzo (see \cite{GN, Pe, GRu}).  We want to use Perazzo polynomials to construct forms $f$ having an arbitrary number of relations among the derivatives. \\
A new kind of forms with vanishing Hessian with a very distinct geometry of Gordan-Noether-Perazzo-Permutti ones were described in recent times (see \cite{GRS}). It is also possible to construct examples with a huge number of relations among the derivatives for this new family (see \cite{GRS}).
\end{rmk}

\begin{example}\rm \label{ex:perazzo} For each degree $d \geq 3$ we construct explicit examples of hypersurfaces of degree $d$ in $\P^4$ non cone having vanishing Hessian. 
For $d=2k+1\geq 3$, let $f=xu^{2k}+yu^kv^k+zv^{2k+1} \in \K[x,y,z,u,v]_d$, then $f_xf_z = f_y^2$ is the only algebraic relation among the derivatives, therefore $\hess_f = 0$ and $\rk \Hess_f = 4 < 5$. 
 For $d=2k \geq 4$, let $f=xu^{2k-1}+yuv^{2k-2}+zu^kv^{k-1} \in \K[x,y,z,u,v]_d$, then $f_xf_y = f_z^2$ is the only algebraic relation among the derivatives, therefore $\hess_f = 0$ and $\rk \Hess_f = 4 < 5$.
In both cases it is easy to see that partial derivatives are not linearly independent, therefore they do not define a cone. 
\end{example}

\begin{definition} \label{justapositionandconcatenation}\rm 
 Let $f \in \K[x_1,\ldots,x_n,u_1,\ldots,u_m]_{(1,d-1)}$, $f= x_1g_1 +\ldots+x_ng_n$ and let $f' \in \K[x'_1,\ldots,x'_r,u'_1,\ldots,u'_s]_{(1,d-1)}$, $f'= x'_1g'_1 +\ldots+x'_rg'_r$ be bigraded polynomials of same bidegree $(1,d-1)$, we define
 $$f\amalg f' = f+f'\in \K[x_1,\ldots,x_n,x'_1,\ldots,x'_r,u_1,\ldots,u_m,u'_1,\ldots,u'_s].$$
 If $g'_1 = (u'_1)^{d-1} $ and $ g_n=u_m^{d-1}$, then we can consider $u'_1=u_m$, $x'_1=x_n$. We define $f \# f' \in \K[x_1,\ldots,x_n,x'_2,\ldots,x'_r,u_1,\ldots,u_m,u'_2,\ldots,u'_m]$, by
 $$f \# f' = f+f'-x'_1(u'_1)^{d-1} = f+f' - x_nu_m^{d-1}.$$
\end{definition}

\begin{prop}\label{prop:gap}
 With the previous notation, we get $$\rk \Hess_{f \amalg f'} = \rk \Hess_{f} + \rk \Hess_{f'} \ \text{and} \ \rk \Hess_{f \# f'} = \rk \Hess_{f} + \rk \Hess_{f'} -2.$$
\end{prop}

\begin{proof}
 By the structure of the Hessians we get $\Hess_{f \amalg f'} = \Hess_{f} \oplus \Hess_{f'}$, then the first claim follows. The second assertion follows from the fact that 
 in this case the co-rank of the Hessians summed up.
\end{proof}

\begin{cor}\label{cor:droprank}
 For any positive integers $\delta$ and $d$, there exists $n$ and a polynomial $f \in \K[x_1,\ldots,x_n]_d$ such that $X = V(f) \subset \P^{n-1}$ is not a cone and $$\rk \Hess_f \leq n-\delta.$$
\end{cor}

\begin{proof}
 
 By the explicit examples \ref{ex:perazzo}, for each $d \geq 3$ there is a form $f \in \K[x_1,\ldots,x_5] $ not defining a cone and such that $\rk \Hess_f = 4$, that is $\delta = 1$.
 Fix $d\geq 3$ and $f_1$ the explicit example of degree $d$ given in \ref{ex:perazzo}. Define $f_{m+1} = f_m \amalg f_1$ using copies of $f$ in a new set of variables. 
 By Proposition \ref{prop:gap}, $\rk \Hess_{f_{m}} = 4m = 5m - m.$
 Therefore $\delta_n = n \to \infty$ and the result follows.  
 \end{proof}

\section{The Jordan type of a graded Artinian Gorenstein algebra}

In this section we prove our main result, that the Jordan type of a standard graded Artinian Gorenstein $\K$-algebra is determined by the rank of certain mixed Hessians. 
This result is, in a certain sense, closely related to \cite[Lemma 3.60]{H-W}. We want to highlight that even though \cite[Lemma 3.60]{H-W} is very elegant, 
our result is more effective as we will see in the next section. 

For a given Artinian Gorenstein $\K$-algebra $A$ and for any linear form $l \in A_1$, consider the short exact sequence:
$$0 \to lA_{i-1} \to A_i \to  A_i/lA_{i-1} \to  0.$$
Since $A_i$ is a finite dimensional $\K$-vector space, up to the choice of a basis completing a basis of $lA_{i-1}$, there is a linear subspace 
$\hat{A}_i \subset A_i$ such that $\hat{A}_i 	\simeq A_i/lA_{i-1}$. 

In the sequel, Definition \ref{def:numberofstrings}, Remark \ref{rmk:1} and Proposition \ref{prop:1} are closely related with the well known concept of Central Simple Modules associated to a graded Artinian Gorenstein algebra (see \cite{HW,H-W}).

\begin{definition}\label{def:numberofstrings} \rm
 Let us define 
$$E_i^j=\{v \in \hat{A}_i|vl^j=0\ \text{and}\ vl^{j-1}\neq 0\},\ j=1,\ldots,d-i.$$
Denote $e_i^j = \dim E_i^j$, $e_j = \displaystyle \sum_{i=0}^d e_i^j$ for $j=1,\ldots,d$  and $e_{d+1} =1$.
\end{definition}

\begin{rmk}\rm \label{rmk:1}First of all note that $A_0 = E_0^{d+1}$, therefore $e_{d+1}=1$.
Note also that choosing $l \in A_1$ general we get $A_d = <l^d>$.
For any $v \in \hat{A}_i$ with $i>d-j \geq 0$ we get $vl^{j-1} = 0 \in A_{i+j-1}$. In fact since $i+j >d$ the only case to consider is $i=d-j+1$, in this case $vl^{j-1}\in A_d =<l^d>$ hence $vl^{j-1} = \lambda l^d $ that yields $v = \lambda l^{d-j+1} = 0 \in \hat{A}_i\simeq A_i/lA_{i-1}$. Therefore $e_i^j = 0$ for $i>d-j$ and the sum $e_j$ can be simplified to $e_j = \displaystyle \sum_{i=0}^{d-j} e_i^j$. 
Since $e_0^d =0$ we get $e_d=0$. 

Note that the strings obtained by elements of $E_i^j$ have length $j$. In fact they are of type $\{v, vl, \ldots, vl^{j-1}\}$ for some $v \in E_i^j$. 
In this context we get a direct sum decomposition of $A_i$ that inductively give us a invariant Jordan decomposition of $A$. This construction can be summarized in the following proposition.

\end{rmk}

\begin{prop}\label{prop:1}
 With the previous notation, the Jordan type of $A$ is:
$$\mathcal{J}_A = \displaystyle\bigoplus_{j=1}^{d+1} j^{e_j}. $$
\end{prop}

\begin{proof}
 By construction, 
$$\displaystyle A_i = lA_{i-1} \oplus \left[\displaystyle \bigoplus_{j=1}^{d-i} E_i^j\right].$$
The number of invariant spaces of dimension $j$ whose element of small degree is $i$ is precisely $e_i^j$. Therefore, 
the number of Jordan blocks of length $j$ is $e_j = \displaystyle \sum_{i=0}^d e_i^j$ and the result follows.
\end{proof}

\begin{lemma} \label{sum} Let $K_i^j$ be the kernel of $\mu_{l^{j}} : A_i \to A_{i+j}$ and let $k_i^j = \dim K_i^j$. Then 
$$\displaystyle \sum_{s=1}^j e_i^s = k_i^j - k_{i-1}^j + k_{i-1}^1.$$
\end{lemma}

\begin{proof}
With the previous notation we get: 
\begin{equation}\label{eq1}
K_i^j \simeq [\displaystyle \bigoplus_{s=1}^j E_i^s]\oplus [lA_{i-1}\cap K_i^j].\end{equation}
On the other hand, we have a natural exact sequence given by multiplication by $l \in A_1$:
\begin{equation}\label{seq_natural}
0 \to K_{i-1}^1 \to A_{i-1} \to lA_{i-1} \to 0.  
\end{equation}
Restricting the sequence \ref{seq_natural} to $K_{i-1}^{j+1} \subset A_{i-1}$ we get:
\begin{equation}
\begin{array}{ccccccccc}
0 & \to & K_{i-1}^1 & \to & A_{i-1} & \to       & lA_{i-1} & \to  & 0\\
  &     &   \cup        &     &    \cup     &   & \cup & & \\
0 & \to & K_{i-1}^1 & \to & K_{i-1}^{j+1} & \to       & lA_{i-1} \cap K_i^j & \to & 0   
\end{array}
\end{equation}
Hence, $lA_{i-1} \cap K_i^j \simeq K_{i-1}^{j+1}/K_{i-1}^1$. Substituting in \ref{eq1} and taking dimensions the result follows.

\end{proof}

\begin{lemma} Let $r_i^j$ be the rank of the multiplication map $\mu_{l^j}:A_i \to A_{i+j}$ for a generic $l \in A_1$. Then
$$e_i^j = r_i^{j-1}-r_i^j-r_{i-1}^j+r_{i-1}^{j+1}.$$
\end{lemma}

\begin{proof}
We will prove by induction in $j$ the following expression: $$e_{i}^{j}=k_{i}^{j}-k_{i}^{j-1}-k_{i-1}^{j+1}+k_{i-1}^{j}.$$
For $j=1$ it follows by Lemma \ref{sum}, since $k_{i}^{0}=0$.\\
Suppose that the result is true for $s<j$. By Lemma \ref{sum}, we get
\begin{equation}
\begin{array}{lll}
e_{i}^{j} &=& k_i^j - k_{i-1}^j + k_{i-1}^1 - \sum_{s=1}^{j-1} e_i^s\\
 & =& k_i^j - k_{i-1}^j + k_{i-1}^1 - \sum_{s=1}^{j-1} [k_{i}^{s}-k_{i}^{s-1}-k_{i-1}^{s+1}+k_{i-1}^{s}] \\
 &=& k_i^j - k_{i-1}^j + k_{i-1}^1 - \sum_{s=1}^{j-1} [k_{i}^{s}-k_{i}^{s-1}]+ \sum_{s=1}^{j-1}[k_{i-1}^{s+1}-k_{i-1}^{s}]\\
 &=& k_i^j - k_{i-1}^j + k_{i-1}^1 - [k_{i}^{j-1}-k_{i}^{0}]+ [k_{i-1}^{j}-k_{i-1}^{1}]\\
  & =& k_i^j   - k_{i}^{j-1} -k_{i-1}^{1}+ k_{i-1}^1
 \end{array}
\end{equation}
The result follows from the relation $\dim A_i = k_i^j + r_i^j$.
\end{proof}

Recall that by Theorem \ref{thm:generalization}, $r_i^j = \rk \Hess_f^{(d-i-j,i)} $, so the Jordan type of $A$ is determined by the rank of these mixed Hessians. 
We can explicitly calculate the Jordan type.

The next result is a Lemma that will be useful for the explicit algorithm given in Corollary \ref{cor:stringdiagram}. 
 We call String diagram when we show all the strings representing a Jordan decomposition of an AG algebra in its Ferrer diagram. 

\begin{lemma}\footnote{This Lemma was reported in a private communication by A. Iarrobino.}
 The string diagrams of AG algebras are symmetric and $$e_i^j = e_{d-i-j+1}^j.$$
\end{lemma}

\begin{proof}
 For the string $\sigma  = \{v, vl, \ldots, vl^{j-1}\}$ with $v_k = vl^k \in A_{i+k}$ and $k=0,\ldots,j-1$, we want to define the dual string $\sigma^*$. For $k=1,\ldots,j$, let $(vl^{j-k})^* \in A^*_{j+i-k}$. Since we can choose $A_d = <l^d>$, the isomorphism $A_{i+j-1}^* \simeq A_{d-i-j+1}$ can be expressed in the following way: there is a $w \in A_{d-i-j+1}$ such that $w(l^{j-1}v)=l^d$. We can associate $(vl^{j-k})(wl^{d+k-j-1})=l^d$, hence we can associate $(vl^{j-k})^*$  to $w_k = wl^{d+k-j-1}$. The dual string of 
$\sigma = \{v_0=v,\ldots,v^{j-1}\}$ is the string $\sigma^* = \{w_0=w, \ldots,w_{j-1} \}$, therefore the diagram is symmetric. Furthermore, $e_i^j =e_{d-i-j+1}^j$ since the number of strings starting in degree $i$ of lenght $j$, by the smetry, is equal to  the number of strings of lenght $j$ that finish in $d-i$. 

\end{proof}

\begin{theorem} \label{thm:main} The Jordan type of any standard graded Artinian Gorenstein $\K$-algebra $A$ presented as $A = Q/\Ann_f$ depends only on the rank of the mixed Hessians of $f$. 
More precisely and explicitly, 
$$\mathcal{J}_A = \displaystyle\bigoplus_{j=1}^{d+1} j^{e_j}$$
with $e_{d+1}=1$ and for $j \leq d$, we have either
\begin{enumerate}
 \item[(i)] $e_j= 2 \displaystyle \sum_{s=0}^m r_s^{j-1} -4 \displaystyle \sum_{s=0}^m r_s^{j} + 2 \displaystyle \sum_{s=0}^{m-1} r_s^{j+1} +2r_m^j$ if $d-j = 2m$ or
 \item[(ii)] $e_j= 2 \displaystyle \sum_{s=0}^m r_s^{j-1} -4 \displaystyle \sum_{s=0}^m r_s^{j} + 2 \displaystyle \sum_{s=0}^{m} r_s^{j+1} +r_{m+1}^{j-1} - r_m^{j+1}$ if $d-j = 2m+1$.
\end{enumerate}

\end{theorem}

\begin{proof}
 \begin{equation}
 \begin{array}{lll}
 e_j & = &\displaystyle \sum_{i=0}^{d-j} e_i^j\\
 &=&\displaystyle\sum_{i=0}^{d-j} [r_i^{j-1}-r_i^j-r_{i-1}^j+r_{i-1}^{j+1}]\\
 &=&\displaystyle\sum_{i=0}^{d-j} r_i^{j-1}-2 \displaystyle\sum_{i=0}^{d-j-1} r_i^j+ \displaystyle\sum_{i=0}^{d-j-2} r_{i}^{j+1}\\
 \end{array}
 \end{equation}
 
By Poincar\'e duality, $r_{i}^{j}=r_{d-i-j}^{j}$, hence we get two cases depending on the parity of $d-j$:
 \begin{enumerate}
 \item[(i)] For $d-j=2m$, we get
 \begin{equation}
 \begin{array}{lll}
 e_j & = &\displaystyle\sum_{i=0}^{2m} r_i^{j-1}-2 \displaystyle\sum_{i=0}^{2m-1} r_i^j+ \displaystyle\sum_{i=0}^{2m-2} r_{i}^{j+1}\\
 &=& 2r_{m}^{j} + 2 \displaystyle\sum_{i=0}^{m} r_i^{j-1}-4 \displaystyle\sum_{i=0}^{m} r_i^j+ 2 \displaystyle\sum_{i=0}^{m-1} r_{i}^{j+1}
 \end{array}
 \end{equation}
 \item[(ii)] For $d-j=2m+1$, we get
 \begin{equation}
 \begin{array}{lll}
 e_j & = &\displaystyle\sum_{i=0}^{2m+1} r_i^{j-1}-2 \displaystyle\sum_{i=0}^{2m} r_i^j+ \displaystyle\sum_{i=0}^{2m-1} r_{i}^{j+1}\\
 &=& r_{m+1}^{j-1}-r_{m}^{j+1} + 2 \displaystyle\sum_{i=0}^{m} r_i^{j-1}-4 \displaystyle\sum_{i=0}^{m} r_i^j+ 2 \displaystyle\sum_{i=0}^{m} r_{i}^{j+1}
 \end{array}
 \end{equation}
 \end{enumerate}
\end{proof}

 We give an algorithm to find the Jordan type of an algebra given explicitly in the form $A = Q/\Ann_f$ by computing the rank of mixed Hessians and using Theorem 
\ref{thm:main} to calculate the numbers $e_i^j$, that depend only on these rank of some mixed Hessians.

\begin{cor}{\bf (String diagrams)}\label{cor:stringdiagram}
Let $A$ be an standard graded Artinian Gorenstein $\K$-algebra with Hilbert vector $\Hilb(A)$. The Jordan type of $A$ can be obtained, algorithmically using Strings in a Ferrer diagram 
of $\Hilb(A)$. 
\end{cor}

\begin{proof} 
 By Definition \ref{def:numberofstrings},  the first step of the algorithm is to do a single string in the basis of the Ferrer diagram, from degree $0$ to degree $d$, of length $d+1$. 
 It corresponds to $e_0^{d+1}=1$ which implies $e_{d+1}=1$, since it is the only string of length $d+1$. By the symmetry of the diagram there is no string of length $d$, {\it i.e.} $e_d=0$.\\
 Use Theorem \ref{thm:main} to calculate all $e_i^j$ for $i \leq \frac{d}{2}$.\
 Suppose that we did all the strings starting in degree $k<i$. In order to construct all the strings that start in $i$, construct, for all $j$ such that $e_i^j \neq 0$, $e_i^j$ strings 
 starting in the level $i$ of the Ferrer diagram of length $j$. It can be justified by Definition \ref{def:numberofstrings}. By the symmetry of the diagram construct now all the dual stings. 
 
\end{proof}

\begin{rmk}\rm  Note that Theorem \ref{thm:main} and Corollary \ref{cor:stringdiagram} can also be used to find the Jordan type of an algebra 
with respect to the multiplication $\mu_l:A \to A$ with $l \in A_1$ any linear form, not necessarily generic. In this case the ranks of Hessians that we consider are the ranks of $\Hess^{(i,j)}_f(l^{\perp})$.
Note also that in this case the basis of the algorithm given in \ref{cor:stringdiagram} changes as soon as $f(l^{\perp})=0$, and it depends on the multiplicity of such a root.
\end{rmk}


\section{Jordan types for algebras of low socle degree }

Note that by Proposition \ref{prop:jordanSLP}, if $A$ has the SLP, then $\J_A = \Hilb(A)^{\vee}$, in this case we say that  $\Delta(A)=0$. If it is not the case, then we say that the $A$ has a gap. We say that a partition $P$ is a AG-partition if there is an AG algebra $A'$ such that $\J_{A'}=P$. We define the gap of $A$, $\Delta(A)$, in the following way. 
$\Delta(A) = \delta$, if $\delta $ is the maximum number of distinct AG-partitions $P_1,\ldots,P_{\delta}$ such that: 

$$\J_A = P_0 \prec \ldots \prec P_{\delta-1} \prec P_{\delta} = \Hilb(A)^{\vee}.$$

The gap of $A$, $\Delta(A)$ is the length of a maximal chain from $\J_A$ to $\Hilb(A)^{\vee}$.

\subsection{Jordan type of algebras of socle degree three}

\begin{prop}\label{prop:cubicsingeneral}
 The Jordan type of an algebra $A=Q/\Ann_f$ with Hilbert vector $(1,n,n,1)$ such that $\rk Hess_f = r\leq n$ is 
 $$\J_A = 4^1\oplus 2^{r-1} \oplus 1^{2(n-r)}.$$
 In this case $\Delta(A) = n-r.$
\end{prop}

\begin{proof}
By Theorem \ref{thm:main} we get: 
 
\begin{enumerate}

\item $e_1 = e_1^1+e_2^1 = 2(n-r)$;
\item $e_2 = e_1^2 = r-1$;
\item $e_3=0$; 
\item $e_4=1$.
\end{enumerate}
Consider the String diagram given by Corollary \ref{cor:stringdiagram}.

 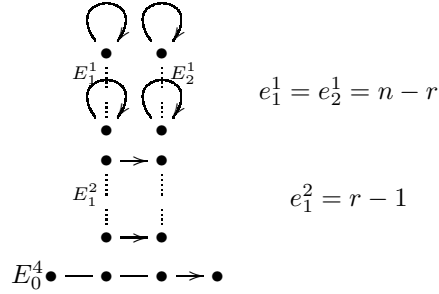
\begin{figure}[H]
 \vspace{0.5cm}
\centerline{\begin{xy}
\xymatrix@R1pt@C10pt{
&\ar@(ul,ur){\bullet} \ar@{..}[d]_{E^1_1}  & \ar@(ul,ur){\bullet}
\ar@{..}[d]^{E^1_2} \\
& \ar@{..}[d]  & 
\ar@{..}[d]  && \hspace{-0.1cm} e_1^1=e_2^1= n-r \\
&\ar@(ul,ur){\bullet}   & \ar@(ul,ur){\bullet}\\
&\ar@{..}[d]_>{E_1^2}{\bullet}  \ar@{->}[r] & {\bullet}\ar@{..}[d] \\
& \ar@{..}[d]  & 
\ar@{..}[d] & & \hspace{-0.1cm} e_1^2=r-1 \\
&{\bullet}  \ar@{->}[r] & {\bullet}\\
{E_0^4}{\bullet}  \ar@{-}[r]&{\bullet}  \ar@{-}[r]&{\bullet}  \ar@{->}[r] & {\bullet}}
\end{xy}}
    \centering
    \caption{String diagram for algebras of socle degree three.}
    \label{fig:my_label}
 \end{figure}

We get $\J_A = 4^1\oplus 2^{r-1} \oplus 1^{2(n-r)}$.
Define $P_k = 4^1\oplus 2^{n-k-1} \oplus 1^{2k}$, then it is easy to see that 
$$\J_A = P_0 \prec P_1 \prec \ldots \prec P_r = \Hilb(A)^{\vee}.$$
It shows that $\Delta(A) \geq n-r$, to conclude the proof use the symmetry of GA partitions to see that it is not possible to obtain a grater chain.  
\end{proof}

By Gordan-Noether Theorem, there is no hypersurfaces in $\P^{n-1}$ with $n \leq 4$ not a cone and having vanishing Hessian (see \cite[Theorem 3.9]{Go}). 
The classification of cubics with vanishing Hessian in low dimension is part of the classical work of Perazzo (see \cite{Pe}) this classical work was revisited in \cite{GRu}. 

\begin{example}\label{example1}\rm
 The first example considered by Perazzo was $X = V(f) \subset \P^4$ with $f=xu^2+yuv+zv^2$. Up to projective transformations it is the only cubic hypersurface with vanishing Hessian not a cone in $\P^4$ (see \cite[Theorem 5.2]{GRu}). 
 The algebra $A = Q/\Ann_f$ has Hilbert vector $\Hilb(A) = (1,5,5,1)$ and Jordan type $\J_A = 4^1\oplus 2^3\oplus 1^2 \prec 4^1 \oplus 2^5$ and $\Delta(A) =1$. 

  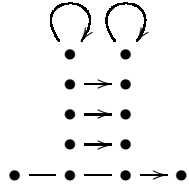
\begin{figure}[H]
 \vspace{0.5cm}
\centerline{\begin{xy}
\xymatrix@R1pt@C10pt{
&\ar@(ul,ur){\bullet}   & \ar@(ul,ur){\bullet} \\
&{\bullet}  \ar@{->}[r] & {\bullet}\\
&{\bullet}  \ar@{->}[r] & {\bullet}\\
&{\bullet}  \ar@{->}[r] & {\bullet}\\
{\bullet}  \ar@{-}[r]&{\bullet}  \ar@{-}[r]&{\bullet}  \ar@{->}[r] & {\bullet}}
\end{xy}}
    \centering
    \caption{String diagram for $\J_A = 4^1\oplus 2^3\oplus 1^2$.}
    \label{fig:my_label}
 \end{figure}
\end{example}

The cubic hypersurfaces $X=V(f) \subset \P^{n-1}$ with $n=5,6,7$ not a cone and having vanishing Hessian were classified by Perazzo and 
the co-rank of $\Hess_f$ is $1$ for all of them (see \cite{Pe} or \cite[Theorem 5.2, 5.3, 5.6, 5.7]{GRu}). \\
We present now explicit examples of cubics with vanishing Hessian whose co-rank of $\Hess_f$ is grater than $1$. 
The following examples use the ideas of Proposition \ref{prop:gap}. A variation of such examples can be found in \cite{GRu}.
 
\begin{example}\label{exe1corank2}\rm For $n=8$, consider $f=xu^2+yuv+zv^2$, $f'$ a copy of $f$ and $g = f\#f'$.
 $$g=x_1u^2+x_2uv+x_3v^2+x_4vw+x_5 w^2$$
 The co-rank of $\Hess_g$ is two. Putting $g_i=\frac{\partial g}{\partial x_i}$, the explicit relations among the derivatives are  $g_1g_3=g_2^2$ and $g_3g_5=g_4^2$. It is easy to check that $V(g)$ is not a cone. Let $A = Q/\Ann)g$, then:
 $$\Hilb(A) = (1,8,8,1)$$
 $$\J_A = 4^1 \oplus 2^5 \oplus 1^4 \prec 4^1 \oplus 2^6 \oplus 1^2 \prec \Hilb(A)^{\vee}.$$
 Therefore, $\Delta(A) =2$. Its String diagram is:

   \begin{figure}[H]
 \vspace{0.5cm}
 \centerline{\begin{xy}
\xymatrix@R1pt@C10pt{
&\ar@(ul,ur){\bullet}   & \ar@(ul,ur){\bullet} \\
&\ar@(ul,ur){\bullet}   & \ar@(ul,ur){\bullet} \\
&{\bullet}  \ar@{->}[r] & {\bullet}\\
&{\bullet}  \ar@{->}[r] & {\bullet}\\
&{\bullet}  \ar@{->}[r] & {\bullet}\\
&{\bullet}  \ar@{->}[r] & {\bullet}\\
&{\bullet}  \ar@{->}[r] & {\bullet}\\
{\bullet}  \ar@{-}[r]&{\bullet}  \ar@{-}[r]&{\bullet}  \ar@{->}[r] & {\bullet}}
\end{xy}}
    \centering
    \caption{String diagram for $\J_A = 4^1 \oplus 2^5 \oplus 1^4$.}
    \label{fig:my_label}
 \end{figure}
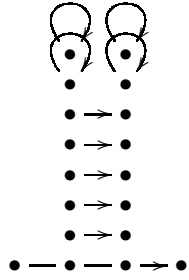
 
\end{example}


\begin{example}\label{exe1corank2}\rm For $n=9$, consider 
 $$f=x_1u_1^2+x_2u_1u_2+x_3u_2^2+x_4u_2u_3+x_5 u_3^2+x_6u_3u_1.$$
 The co-rank of $\Hess_f$ is $3$. Putting $f_i=\frac{\partial f}{\partial x_i}$, the explicit relations among the derivatives are  $f_1f_3=f_2^2$, $f_3f_5=f_4^2$ and $f_5f_1=f_6^2$ and they are algebraically independent. It is easy to check that $V(f)$ is not a cone. Let $A = Q/\Ann_f$, then:
 $$\Hilb(A) = (1,9,9,1)$$
 $$\J_A = 4^1 \oplus 2^5 \oplus 1^6 \prec 4^1 \oplus 2^6 \oplus 1^4 \prec 4^1 \oplus 2^7 \oplus 1^2 \prec \Hilb(A)^{\vee}.$$
 Therefore, $\Delta(A) =3$. Its String diagram is:

   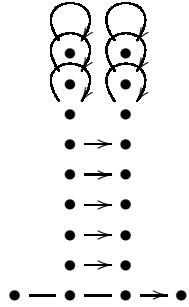
\begin{figure}[H]
 \vspace{0.5cm}
\centerline{\begin{xy}
\xymatrix@R1pt@C10pt{
&\ar@(ul,ur){\bullet}   & \ar@(ul,ur){\bullet} \\
&\ar@(ul,ur){\bullet}   & \ar@(ul,ur){\bullet} \\
&\ar@(ul,ur){\bullet}   & \ar@(ul,ur){\bullet} \\
&{\bullet}  \ar@{->}[r] & {\bullet}\\
&{\bullet}  \ar@{->}[r] & {\bullet}\\
&{\bullet}  \ar@{->}[r] & {\bullet}\\
&{\bullet}  \ar@{->}[r] & {\bullet}\\
&{\bullet}  \ar@{->}[r] & {\bullet}\\
{\bullet}  \ar@{-}[r]&{\bullet}  \ar@{-}[r]&{\bullet}  \ar@{->}[r] & {\bullet}}
\end{xy}}
    \centering
    \caption{String diagram for $\J_A = 4^1 \oplus 2^5 \oplus 1^6$.}
    \label{fig:my_label}
 \end{figure}
\end{example}


\begin{theorem} \label{thm:cubics}
 Let $A$ be a standard graded Artinian Gorenstein $\K$-algebra of Hilbert vector $\Hilb(A) = (1,n,n,1)$. Then the possibles Jordan types of $A$ are:
 \begin{enumerate}
  \item[(i)] For $n \leq 4$, $A$ has the SLP, therefore $\J_A = 4^1 \oplus2^{n-1}$, {\it i.e.} $\Delta(A)=0$;    
  \item[(ii)] For $n \in \{5,6, 7\}$, either $A$ has the SLP and $\J_A = 4^1 \oplus2^{n-1}$ or $A$ fails the SLP and  $\J_A = 4^1 \oplus2^{n-2} \oplus 1^2$, {\it i.e.} $\Delta(A)\leq1$;  
  \item[(iii)] For $n =8$, either $A$ has the SLP and $\J_A = 4^1 \oplus2^7$ or $A$ fails the SLP and $\J_A = 4^1 \oplus2^6 \oplus 1^2$ or $\J_A = 4^1 \oplus2^5 \oplus 1^4$, {\it i.e.} $\Delta(A)\leq2$;
  \item[(iv)] For $n =9$, either $A$ has the SLP and $\J_A = 4^1 \oplus2^8$ or $A$ fails the SLP and $\J_A = 4^1 \oplus2^7 \oplus 1^2$ or $\J_A = 4^1 \oplus2^6 \oplus 1^4$ or $\J_A = 4^1 \oplus2^5 \oplus 1^6$, {\it i.e.} $\Delta(A)\leq3$;
  \item[(v)] For any positive integer $\delta$, there are $n$ and $A$ such that $\Delta(A) = \delta$.
 \end{enumerate}

\end{theorem}

\begin{proof} By Proposition \ref{prop:jordanSLP}, $\J_A = \Hilb(A)^{\vee}$ if and only if $A$ has the SLP. By Theorem \ref{thm:generalization}, an algebra $A = Q/\Ann_f$ of socle degree $3$ has the SLP if and only if $\hess_f \neq 0$.  
 Since in this case the only Hessian that matters is the first, the classical Hessian, we can use Gordan-Noether theory (see \cite[Chapter 7]{Ru} or \cite[Appendix A]{Go}).
 By Theorem \ref{thm:main}, we have to find $f \in R_3$ with vanishing Hessian to get a Jordan type distinct to the dual of the Hilbert vector. 
 By \cite[Theorem 1]{DP} we can reduce ourselves to the irreducible case (see also \cite[Theorem 3.5]{Go}). By Gordan-Noether Theorem \cite[Theorem 3.9]{Go}, for $n \leq 4$, it is not possible, hence (i) follows. \\
 By Theorem \cite[Theorem 5.2, 5.3, 5.6, 5.7]{GRu}, for $n \in \{5,6, 7\}$, the co-rank of the Hessian is $1 $ for all irreducible cubics with vanishing Hessian. By Theorem \ref{thm:main}, the result follows. \\
 For $n = 8$, the co-rank of the Hessian is less then or equal to $2$, for $n=9$ it is less than or equal to $3$. These cases are possible by Examples \ref{exe1corank2} and \ref{exe1corank3} In fact, if $X = V(f) \subset \P^n $ with $n=7,8$, and the co-rank of the Hessian is grater than one, then choosing a point outside the polar image and projecting from the right 
 and taking hyperplane section on the left we get can consider $X' = V(f') \subset \P^{n-1}$ with vanishing Hessian and the co-rank of $\Hess_{f'}$ is one less, but it is not possible by Theorem \cite[Theorem 5.2, 5.3, 5.6, 5.7]{GRu}.
 The result follows from Theorem \ref{thm:main}.\\
 The last assertion follows from Corollary \ref{cor:droprank} and by Proposition \ref{prop:cubicsingeneral}.
 \end{proof}


\subsection{Jordan type of algebras of socle degree four}

\begin{prop}\label{prop:quarticsingeneral}
 The Jordan type of an algebra $A=Q/\Ann_f$ with Hilbert vector $(1,n,a,n,1)$ such that $\rk Hess_f = r\leq n$ and $$\rk Hess^{(1,2)}_f = s\leq n$$ is 
$$\J_A = 5^1 \oplus 3^{r-1} \oplus 2^{2(n-r)}\oplus 1^{a-2n+r}.$$
\end{prop}

\begin{proof}
 Let $A = Q/\Ann_f$ with Hilbert vector $H_A = (1,n,a,n,1)$. Consider $rk(Hess_f^1)=r \leq n$ and $rk(Hess^{(1,2)})=s \leq min\{n,a\}$.

By the Theorem \ref{thm:main} we get: 
 
\begin{enumerate}

\item $e_1 = e_1^1+e_2^1+e_3^1 = 2n+a-4s+r$;
\item $e_2 = e_1^2 + e_2^2= 2(s-r)$;
\item $e_3= e_1^3=r-1$;
\item $e_4=0$;
\item $e_5=1$.
\end{enumerate}

Consider the string diagram:

\begin{figure}[H]
 \vspace{0.8cm}
\centerline{\begin{xy}
\xymatrix@R1pt@C10pt{
& &\ar@(ul,ur){\bullet}
\ar@{..}[d]&\\
& \ar@(ul,ur){\bullet}
\ar@{..}[d]_{E^1_1} & 
\ar@{..}[d]_{E^1_2} & \ar@(ul,ur){\bullet}
\ar@{..}[d]_{E^1_3} && \hspace{-0.5cm}  e_1^1=e_3^1=n-s\\
& \ar@{..}[d]  & 
\ar@{..}[d]& \ar@{..}[d] & & \hspace{-0.5cm}  e_2^1= a-2s+r \\
&  \ar@(ul,ur){\bullet}
 & \ar@(ul,ur){\bullet} &  \ar@(ul,ur){\bullet}\\
&  & {\bullet}
\ar@{..}[d]_>{E^2_2}  \ar@{->}[r]& {\bullet}\ar@{..}[d] &  \\
&   & 
\ar@{..}[d] & \ar@{..}[d] && \hspace{-0.5cm}  e_2^2=s-r \\
&   & {\bullet}\ar@{->}[r] &  {\bullet} \\
&\ar@{..}[d]_>{E_1^2}{\bullet}  \ar@{->}[r] & {\bullet}\ar@{..}[d] &  \\
& \ar@{..}[d]  & \ar@{..}[d] & & &\hspace{-0.5cm} e_1^2=s-r\\
&{\bullet}  \ar@{->}[r] & {\bullet}& \\
&\ar@{..}[d]_>{E_1^3}{\bullet}  \ar@{-}[r] & {\bullet} \ar@{->}[r] \ar@{..}[d] & {\bullet}\ar@{..}[d] &\\
& \ar@{..}[d]  & 
\ar@{..}[d] & \ar@{..}[d] & &\hspace{-0.5cm} e_ 1^3=r-1 \\
&{\bullet}  \ar@{-}[r] & {\bullet} \ar@{->}[r] & {\bullet} &\\
{E_0^5}{\bullet}  \ar@{-}[r]&{\bullet}  \ar@{-}[r]&{\bullet}  \ar@{-}[r] &{\bullet}  \ar@{->}[r]& {\bullet}}
\end{xy}}
    \centering
    \caption{String diagram for socle degree four.}
    \label{fig:my_label}
 \end{figure}

We have $\J_A = 5^1\oplus 3^{r-1} \oplus 2^{2(s-r)}\oplus 1^{2n+a-4s+r}$.

\end{proof}

\begin{cor}\label{cor:wlp4}
 The Jordan type of an algebra $A=Q/\Ann_f$ having the WLP with Hilbert vector $(1,n,a,n,1)$ such that $\rk Hess_f = r\leq n$ is 
$$\J_A = 5^1 \oplus 3^{r-1} \oplus 2^{2(n-r)}\oplus 1^{a-2n+r}.$$
\end{cor}

\begin{proof}

Let $A $ be an algebra with Hilbert vector $H_A = (1,n,a,n,1)$. The WLP condition 
implies $a \geq n$ and $r_1^1 = r_2^1 = n$.
Let us suppose that $r_1^2=r$.

By Theorem \ref{thm:main} we get:

\begin{enumerate}
\item $e_1=e_1^1+e_2^1+e_3^1=0+a-2n+r$;
\item $e_2 = e_1^2+e_2^2=(n-r)+(n-r)=2(n-r)$;
\item $e_3=e_1^3=r-1$; 
\item $e_4=0$; 
\item $e_5=e_0^5=1$.
\end{enumerate}

We have $\J_A = 5^1 \oplus 3^{r-1} \oplus 2^{2(n-r)}\oplus 1^{a-2n+r}$.

Consider the string diagram:

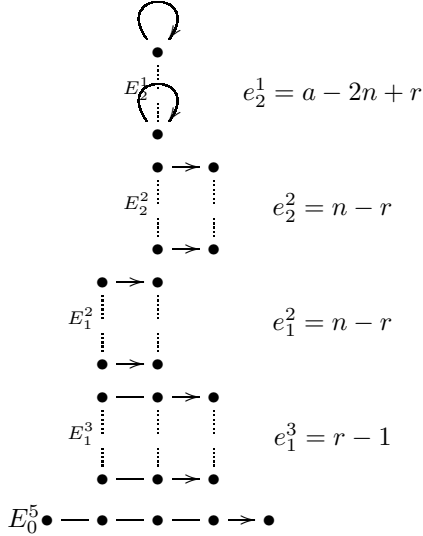
\begin{figure}[H]
 \vspace{0.5cm}
\centerline{\begin{xy}
\xymatrix@R2pt@C10pt{
&  & \ar@(ul,ur){\bullet}
\ar@{..}[d]_>{E^1_2} & \\
&   & 
\ar@{..}[d]&&& \hspace{-1cm} e_2^1=a-2n+r \\
&   & \ar@(ul,ur){\bullet} & \\
&  & {\bullet}
\ar@{..}[d]_>{E^2_2}  \ar@{->}[r]& {\bullet}
\ar@{..}[d] \\
&   & 
\ar@{..}[d] & \ar@{..}[d] && \hspace{-1cm} e_2^2=n-r \\
&   & {\bullet}\ar@{->}[r] & {\bullet} \\
&\ar@{..}[d]_>{E_1^2}{\bullet}  \ar@{->}[r] & {\bullet}\ar@{..}[d] &  \\
& \ar@{..}[d]  & \ar@{..}[d] &  & & \hspace{-1cm} e_1^2=n-r\\
&{\bullet}  \ar@{->}[r] & {\bullet}&\\
&\ar@{..}[d]_>{E_1^3}{\bullet}  \ar@{-}[r] & {\bullet} \ar@{->}[r] \ar@{..}[d] & {\bullet}\ar@{..}[d] &\\
& \ar@{..}[d]  & 
\ar@{..}[d] & \ar@{..}[d] & &\hspace{-1cm} e_1^3=r-1 \\
&{\bullet}  \ar@{-}[r] & {\bullet} \ar@{->}[r] & {\bullet} &\\
{E_0^5}{\bullet}  \ar@{-}[r]&{\bullet}  \ar@{-}[r]&{\bullet}  \ar@{-}[r] &{\bullet}  \ar@{->}[r]& {\bullet}}
\end{xy}}
    \centering
    \caption{String diagram for socle degree four with WLP.}
    \label{fig:my_label}
 \end{figure}
 
\end{proof}

\begin{rmk}\rm Corollary \ref{cor:wlp4} is telling us that that are a huge number, depending on $r$, of intermediate algebras between WLP and SLP.
 
\end{rmk}

\begin{example}\label{exe3corank2} \rm For $n=8$, consider $A = Q/\Ann_f$ with 
 $$f=x_1u^2v+x_2uv^2+x_3u^3+x_4uw^2+x_5 u^2w.$$
 It is easy to see that $\Hilb(A) = (1,8,10,8,1)$ and that $A$ has the WLP. On the other hand $r = \rk \Hess_f =6$. In fact putting $f_i =\frac{\partial f}{\partial x_i}$, we get the following explicit relations among the derivatives 
 $f_2f_3=f_1^2\ \ \text{and}\ \ f_3f_4=f_5^2$. Therefore:
 $$\J_A = 5^1 \oplus 3^5 \oplus 2^4 \prec 5^1 \oplus 3^6 \oplus 2^2 \oplus 1^1 \prec \Hilb(A)^{\vee},\ \Delta(A)=2.$$
 
 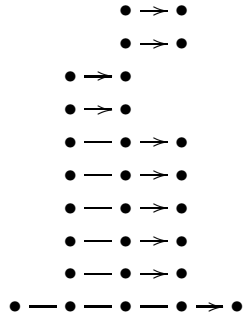
\begin{figure}[H]
 \vspace{0.5cm}
\centerline{\begin{xy}
\xymatrix@R2pt@C10pt{
& &{\bullet}  \ar@{->}[r] & {\bullet}&\\
&   & {\bullet}\ar@{->}[r] & {\bullet} \\
&{\bullet}  \ar@{->}[r] & {\bullet}&\\
&{\bullet}  \ar@{->}[r] & {\bullet}&\\
&{\bullet}  \ar@{-}[r] & {\bullet} \ar@{->}[r] & {\bullet} &\\ 
&{\bullet}  \ar@{-}[r] & {\bullet} \ar@{->}[r] & {\bullet} &\\
&{\bullet}  \ar@{-}[r] & {\bullet} \ar@{->}[r] & {\bullet} &\\
&{\bullet}  \ar@{-}[r] & {\bullet} \ar@{->}[r] & {\bullet} &  \\
&{\bullet}  \ar@{-}[r] & {\bullet} \ar@{->}[r] & {\bullet} &\\
{\bullet}  \ar@{-}[r]&{\bullet}  \ar@{-}[r]&{\bullet}  \ar@{-}[r] &{\bullet}  \ar@{->}[r]& {\bullet}}
\end{xy}}
    \centering
    \caption{String diagram for $\J_A = 5^1 \oplus 3^5 \oplus 2^4 \prec 5^1 \oplus 3^6 \oplus 2^2 \oplus 1^1$.}
 \end{figure}

\end{example}


\begin{theorem} \label{thm:quartics}
 Let $A$ be a standard graded Artinian Gorenstein $\K$-algebra of Hilbert vector $\Hilb(A) = (1,n,a,n,1)$. Then the possibles Jordan types of $A$ are:
 \begin{enumerate}
  \item[(i)] If $n \leq 4$, then $A$ has the SLP, therefore $\J_A = 5^1 \oplus3^{n-1} \oplus 1^{a-n} \Hilb(A)^{\vee}$, {\it i.e.}, $\Delta(A)=0$;    
  \item[(ii)] If $n = 5$, then either $A$ has the SLP and $\J_A = 5^1 \oplus3^{n-1} \oplus 1^{a-n} = \Hilb(A)^{\vee}$ or $A$ fails SLP but has the WLP and $\J_A = 5^1 \oplus3^{n-2} \oplus 2^2 \oplus 1^{a-n-1}$, {\it i.e.}, $\Delta(A) \leq 1$ ;
    \item[(iii)] If $n \geq 6$, then there are algebras failing WLP. 
    \item[(iv)] For any positive integer $\delta$, there are $n$ and $A$ such that $\Delta(A) = \delta$.

\end{enumerate}
  \end{theorem}
  
  \begin{proof} Let $A =Q/\Ann_f$ with $f \in \K[x_1,\ldots,x_n]_4$. 
   By Theorem \ref{thm:generalization}, for $l \in A_1$, the multiplication map $\mu_{l^2} :A_1 \to A_3$ has rank 
   $\rk \Hess_f(l^{\perp})$. For $n \leq 4$, by Gordan Noether Theorem, we get $\hess_f \neq 0$. Therefore, for a general $l \in A_1$, the map $\mu_{l^2} :A_1 \to A_3$
is an isomorphism, implying that $A$ has the SLP. Hence  $\J_A = 5^1 \oplus3^{n-1} \oplus 1^{a-n}$.\\
 The second assertion, for $n=5$, uses the same arguments of \cite[Thm. 3.5]{Go}. If $\hess_f \neq 0$, then $A$ has the SLP and the result follows. We recall that hypersurfaces of degree $4$ with vanishing Hessian are classified by Gordan Noether Theorem (see \cite{GN,CRS,GR} or \cite[Chapter 7]{Ru}). 
  If $\hess_f = 0$, by Theorem \cite[Thm. 1]{DP}, we can reduce ourselves to the reduced case. 
  By the classification Theorem of Gordan and Noether in $\P^4$, Theorem (see \cite{GN,CRS,GR} or \cite[Chapter 7]{Ru}), up to a projective transformation $f$ is be of the form
 $$f=x_1f_1+x_2f_2+x_3f_3+h \in\K[x_1,x_2,x_3,u,v]_4,$$
 with $f_i \in \K[u,v]_3$ and $h \in \K[u,v]_4$. We can suppose that $f$ is irreducible and conclude that $r=\rk \Hess_f=4$.\\
 Consider the map $\phi:\P^1 \dashrightarrow \P^2$ given by $\phi(u:v)=(f_1:f_2:f_3)$.
 The image of $\phi$, $Z = \overline{\phi(\P^1)}$ is a rational curve of degree two or three. In fact, it is a projection of the twisted cubic $\mathcal{V}_3(\P^1) \subset \P^3$
 from a point. We have only three possibilities:
 \begin{enumerate}
  \item[(i)] Projection from an internal point. In this case  $$f=x_1u^3+x_2uv^2+x_3u^2v+h(u,v).$$
  Therefore $A$ satisfy the WLP, with $l=U+V \in A_1$.
 
 \item[(ii)] An external projection whose center belongs to the tangent surface of the twisted cubic, $T \mathcal{V}_3(\P^1)$. In this case  
 $$f=x_1u^2v+x_2u^3+x_3v^3+h(u,v).$$
  Therefore $A$ satisfy the WLP, with $l=U+V \in A_1$.
 \item[(iii)] A general external projection. In this case $Z \subset \P^2$ is a nodal cubic curve.  In this case  
  $$f=x_1v(u^2-v^2)+x_2u(u^2-v^2)+x_3v^3+h(u,v).$$
  Therefore $A$ satisfy the WLP, with $l=V \in A_1$.
\end{enumerate}
Assertion (iii) is part of Theorem \cite[Thm. 3.8]{Go}. The last assertion follows from Corollary \ref{cor:droprank}.
  \end{proof}

\subsection{Jordan type of algebras of socle degree five }

\begin{prop}\label{prop:quintics}
 The Jordan type of an algebra $A=Q/\Ann_f$ with Hilbert vector $(1,n,a,a,n,1)$ such that $\rk Hess_f=r_1^3 = r\leq n$, $\rk Hess^{(1,2)}_f =r_1^2=p\leq n$, $\rk Hess_f^2=r_2^1 = q$ and $\rk Hess^{(1,3)}_f =r_1^1=s$ is 
$$\J_A = 6^1 \oplus 4^{r-1} \oplus 3^{2(p-r)}\oplus 2^{2s-4p+q+r}\oplus 1^{2n+2a-4s-2q+2p} .$$
\end{prop}

\begin{proof}

By the Theorem \ref{thm:main} we get: 
 
\begin{enumerate}
\item $e_1=e_1^1+e_2^1+e_3^1+e_4^1=2n+2a-4s-2q+2p$;
\item $e_2 = e_1^2+e_2^2+e_3^2=2s-4p+q+r$;
\item $e_3=e_1^3+e_2^3=2(p-r)$;
\item $e_4=e_1^4=r-1$; 
\item $e_5=0$; 
\item $e_6=e_0^6=1$.
\end{enumerate}

Consider the string diagram:

\begin{figure}[H]
 \vspace{0.5cm}
\centerline{\begin{xy}
\xymatrix@R1pt@C10pt{
&  & \ar@(ul,ur){\bullet}
\ar@{..}[d]_{E^1_2} & \ar@(ul,ur){\bullet}
\ar@{..}[d]_{E^1_3} &   \\
&   & 
\ar@{..}[d]& \ar@{..}[d] & & &\hspace{-0.5cm}  e_2^1=e_3^1= a-s-q+p \\
&   & \ar@(ul,ur){\bullet} &  \ar@(ul,ur){\bullet}
&  \\
& \ar@(ul,ur){\bullet}
\ar@{..}[d]_{E^1_1} & \ar@{..}[d]_{E^2_2} {\bullet}  \ar@{->}[r]& {\bullet}
\ar@{..}[d] & \ar@(ul,ur){\bullet}
\ar@{..}[d]_{E_4^1}\\
&  \ar@{..}[d] & 
\ar@{..}[d]&\ar@{..}[d]& \ar@{..}[d]& &\hspace{-0.5cm} e_ 1^1=e_2^2=e_4^1=n-s \\
& \ar@(ul,ur){\bullet}  & {\bullet}  \ar@{->}[r]& {\bullet}& \ar@(ul,ur){\bullet} \\
&\ar@{..}[d]_>{E_1^2}{\bullet}  \ar@{->}[r] & {\bullet} \ar@{..}[d] & {\bullet} \ar@{->}[r]\ar@{..}[d]_>{E^2_3} &\ar@{..}[d]  {\bullet} \\
& \ar@{..}[d]  & 
\ar@{..}[d] & \ar@{..}[d] &\ar@{..}[d] & & \hspace{-0.5cm} e_ 1^2=e_3^2=s-p \\
&{\bullet}  \ar@{->}[r] & {\bullet}  & {\bullet} \ar@{->}[r] & {\bullet}&\\
&  & {\bullet}
\ar@{..}[d]_>{E^3_2}  \ar@{-}[r] &{\bullet}
\ar@{..}[d]  \ar@{->}[r]& {\bullet}\ar@{..}[d] &   \\
& & \ar@{..}[d]  & 
\ar@{..}[d] & \ar@{..}[d] & & \hspace{-0.5cm} e_2^3=p-r \\
&  & {\bullet}  \ar@{-}[r]& {\bullet}\ar@{->}[r] &  {\bullet} \\
&\ar@{..}[d]_>{E_1^3}{\bullet}  \ar@{-}[r] & {\bullet}\ar@{..}[d]\ar@{->}[r] & {\bullet} \ar@{..}[d]&  \\
& \ar@{..}[d]  & \ar@{..}[d] & \ar@{..}[d] & & & \hspace{-0.5cm} e_1^3=p-r\\
&{\bullet}  \ar@{-}[r] & {\bullet} \ar@{->}[r] & {\bullet} & \\
&\ar@{..}[d]_>{E_1^4}{\bullet}  \ar@{-}[r] & {\bullet} \ar@{-}[r] \ar@{..}[d] & {\bullet} \ar@{->}[r]\ar@{..}[d] & {\bullet}\ar@{..}[d] \\
& \ar@{..}[d]  & 
\ar@{..}[d] & \ar@{..}[d] &\ar@{..}[d] & & \hspace{-0.5cm} e_ 1^4=r-1 \\
&{\bullet}  \ar@{-}[r] & {\bullet} \ar@{-}[r] & {\bullet} \ar@{->}[r] & {\bullet}&\\
{E_0^6}{\bullet}  \ar@{-}[r]&{\bullet}  \ar@{-}[r]&{\bullet}  \ar@{-}[r] &{\bullet}  \ar@{-}[r]& {\bullet} \ar@{->}[r]& {\bullet}}
\end{xy}}
    \centering
    \caption{String diagram for socle degree five.}
    \label{fig:my_label}
 \end{figure}
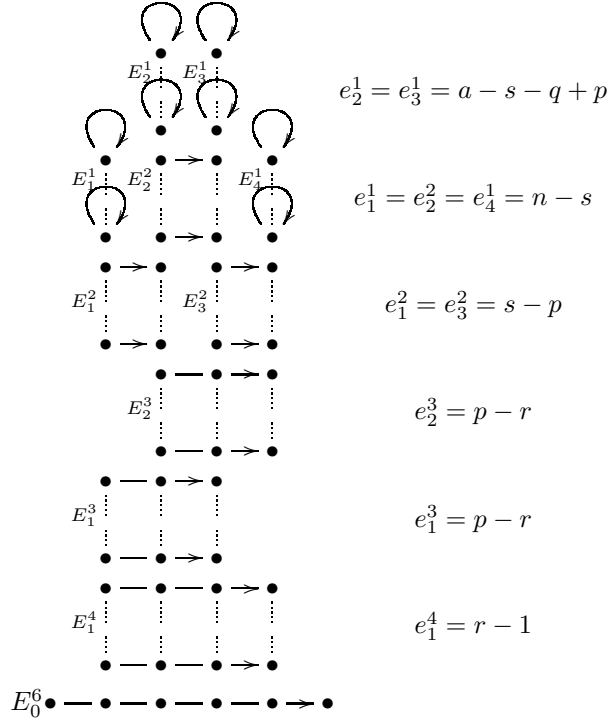

We have $\J_A = 6^1 \oplus 4^{r-1} \oplus 3^{2(p-r)}\oplus 2^{2s-4p+q+r}\oplus 1^{2n+2a-4s-2q+2p} $.

\end{proof}

\begin{cor}\label{cor:quinticsWLP}
 The Jordan type of an algebra $A=Q/\Ann_f$ having the WLP with Hilbert vector $(1,n,a,a,n,1)$ such that $\rk Hess_f = r\leq n$ is 
$$\J_A = 6^1 \oplus 4^{r-1} \oplus 3^{2(n-r)}\oplus 2^{a-2n+r}.$$
\end{cor}

\begin{proof}
Let $A = Q/\Ann_f$ with Hilbert vector $H_A = (1,n,a,a,n,1)$. The WLP condition 
implies $a \geq n$ and $r_1^1 = r_3^1 =r_1^2= r_2^2= n, $ $ r_2^1 = a$.
Let us suppose that $r_1^3=r$.

By Theorem \ref{thm:main} we get:

\begin{enumerate}
\item $e_1=e_1^1+e_2^1+e_3^1+e_4^1=0$;
\item $e_2 = e_1^2+e_2^2+e_3^2=a+r-2n$;
\item $e_3=e_1^3+e_2^3=2(n-r)$;
\item $e_4=e_1^4=r-1$; 
\item $e_5=0$; 
\item $e_6=e_0^6=1$.
\end{enumerate}

By Theorem \ref{thm:main}, $\J_A = 6^1 \oplus 4^{r-1} \oplus 3^{2(n-r)}\oplus 2^{a-2n+r}$.

We have the following string diagram:

\begin{figure}[H]
 \vspace{0.5cm}
\centerline{\begin{xy}
\xymatrix@R1pt@C10pt{
&  & \ar@{..}[d]_{E^2_2} {\bullet}  \ar@{->}[r]& {\bullet}
\ar@{..}[d] & \\
&   & 
\ar@{..}[d]&\ar@{..}[d]& & &\hspace{-0.5cm} e_2^2=a+r-2n \\
&  & {\bullet}  \ar@{->}[r]& {\bullet}&  \\
&  & {\bullet}
\ar@{..}[d]_>{E^3_2}  \ar@{-}[r] &{\bullet}
\ar@{..}[d]  \ar@{->}[r]& {\bullet}\ar@{..}[d] &   \\
& & \ar@{..}[d]  & 
\ar@{..}[d] & \ar@{..}[d] & & \hspace{-0.5cm} e_2^3=n-r \\
&  & {\bullet}  \ar@{-}[r]& {\bullet}\ar@{->}[r] &  {\bullet} \\
&\ar@{..}[d]_>{E_1^3}{\bullet}  \ar@{-}[r] & {\bullet}\ar@{..}[d]\ar@{->}[r] & {\bullet} \ar@{..}[d]&  \\
& \ar@{..}[d]  & \ar@{..}[d] & \ar@{..}[d] & & & \hspace{-0.5cm} e_1^3=n-r\\
&{\bullet}  \ar@{-}[r] & {\bullet} \ar@{->}[r] & {\bullet} & \\
&\ar@{..}[d]_>{E_1^4}{\bullet}  \ar@{-}[r] & {\bullet} \ar@{-}[r] \ar@{..}[d] & {\bullet} \ar@{->}[r]\ar@{..}[d] & {\bullet}\ar@{..}[d] \\
& \ar@{..}[d]  & 
\ar@{..}[d] & \ar@{..}[d] &\ar@{..}[d] & & \hspace{-0.5cm} e_1^4=r-1 \\
&{\bullet}  \ar@{-}[r] & {\bullet} \ar@{-}[r] & {\bullet} \ar@{->}[r] & {\bullet}&\\
{E_0^6}{\bullet}  \ar@{-}[r]&{\bullet}  \ar@{-}[r]&{\bullet}  \ar@{-}[r] &{\bullet}  \ar@{-}[r]& {\bullet} \ar@{->}[r]& {\bullet}}
\end{xy}}
    \centering
    \caption{String diagram for socle degree five with WLP.}
    \label{fig:my_label}
 \end{figure}
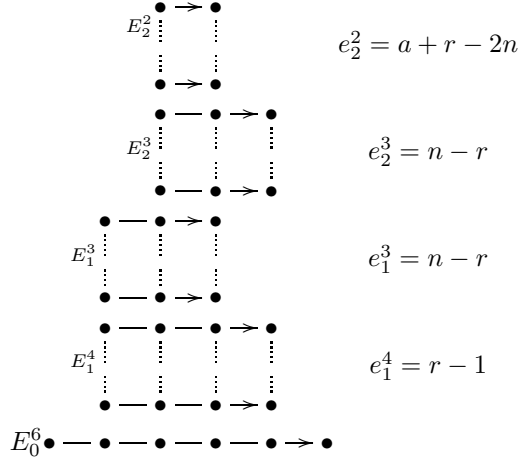

\end{proof}

\begin{example}\rm  The next example is a simplification of a classical example given by Ikeda (see \cite{Ik, MW}).
Consider the algebra $A = Q/\Ann_f$ with
$$f=xu^3v+yuv^3$$
 $$\Hilb(A) = (1,4,7,7,4,1)$$
 It is easy to see that $\hess_f \neq 0$ but $\hess^2_f=0$. Hence 
 $$\J_A = 6^1 \oplus 3^4 \oplus 2^2 \oplus 1^2 \prec \Hilb(A)^{\vee}, \ \Delta(A)=1.$$
 
 \begin{figure}[H]
 \vspace{0.5cm}
\centerline{\begin{xy}
\xymatrix@R1pt@C10pt{
&   & \ar@(ul,ur){\bullet} &  \ar@(ul,ur){\bullet}
&  \\
&  & {\bullet}  \ar@{->}[r]& {\bullet}& \\
&  & {\bullet}  \ar@{->}[r]& {\bullet}&  \\
&  & {\bullet}  \ar@{-}[r]& {\bullet}\ar@{->}[r] &  {\bullet} \\
&  & {\bullet}  \ar@{-}[r]& {\bullet}\ar@{->}[r] &  {\bullet} \\
&{\bullet}  \ar@{-}[r] & {\bullet} \ar@{->}[r] & {\bullet} &\\
&{\bullet}  \ar@{-}[r] & {\bullet} \ar@{->}[r] & {\bullet} & \\
{\bullet}  \ar@{-}[r]&{\bullet}  \ar@{-}[r]&{\bullet}  \ar@{-}[r] &{\bullet}  \ar@{-}[r]& {\bullet} \ar@{->}[r]& {\bullet}}
\end{xy}}
    \centering
    \caption{String diagram for $\J_A = 6^1 \oplus 3^4 \oplus 2^2 \oplus 1^2$}
 \end{figure}
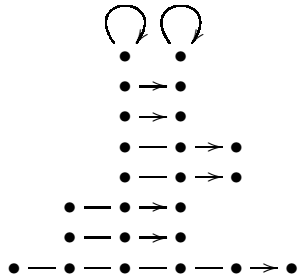

\end{example}


{\bf Acknowledgments}.
We wish to thank Anthony Iarrobino and Junzo Watanabe for his very useful comments and suggestions in a previous version of these notes. The second author would like to thank also
the organizers of the workshop on Lefschetz properties and Jordan Types in Algebra, Geometry and Combinatorics - Levico- Italy 2018.
The second author have a pleasant stay discussing Lefschetz properties and Jordan types there. 
The second author was partially supported  by ICTP-INdAM Research in Pairs Fellowship 2018/2019 and by FACEPE ATP - 0005-1.01/18.

\end{document}